\documentclass[12pt,leqno,fleqn]{nyjm}  
\usepackage{amsmath,amstext,amsthm,amssymb,amsxtra}

\usepackage{mathtools}
\mathtoolsset{showonlyrefs,showmanualtags}

\usepackage{hyperref} 
\hypersetup{
    colorlinks=true,       
    linkcolor=blue,          
    citecolor=magenta,        
    filecolor=magenta,      
    urlcolor=cyan           
}

\usepackage[backrefs]{amsrefs}

\theoremstyle{plain} 
\newtheorem{lemma}[equation]{Lemma} 
\newtheorem{proposition}[equation]{Proposition} 
\newtheorem{theorem}[equation]{Theorem} 
\newtheorem{corollary}[equation]{Corollary} 
\newtheorem{conjecture}[equation]{Conjecture}
\newtheorem{priorResults}{Theorem}

\theoremstyle{definition}

\theoremstyle{remark}

\numberwithin{equation}{section}
\def\theequation{\thesection.\arabic{equation}}

\title[Oscillatory dominated by Sparse] {Sparse Bounds for Oscillatory \\ and Random Singular Integrals}
\author{Michael T. Lacey}   

\address{ School of Mathematics, Georgia Institute of Technology, Atlanta GA 30332, USA}
\email {lacey@math.gatech.edu}
\thanks{Research supported in part by grant NSF-DMS 1265570 and  NSF-DMS-1600693}

\author{Scott Spencer}   

\address{ School of Mathematics, Georgia Institute of Technology, Atlanta GA 30332, USA}
\email {spencer@math.gatech.edu}

\begin{document}

\begin{abstract}
Let  $ T _{P } f (x) =   \int e ^{i P (y)} K (y) f (x-y) \; dy $, where 
$ K (y)$ is a smooth Calder\'on-Zygmund kernel on $ \mathbb R ^{n}$, and $ P$ be a polynomial.  
We show that there is a sparse bound for the bilinear form $ \langle T_P f, g \rangle$. 
This in turn easily implies  $ A_p $ inequalities.  
The method of proof is applied in a random discrete setting, yielding the first weighted inequalities  for  operators defined  on sparse sets of integers. 
\end{abstract}

\maketitle

\section{Introduction} 
Singular integral operators can be pointwise dominated by  sparse operators, which are positive localized operators, something that singular integrals are not.  This paper extends this theme to the settings of (a) oscillatory singular integrals, and (b) discrete random operators.  In both cases, we easily derive weighted inequalities. In the latter case, these are the first such weighted inequalities known.  We state our results before providing a broader context.

 Call a collection of cubes  $ \mathcal S  $ in $ \mathbb R ^{n}$ a \emph{sparse} collection if 
 there is a set $ E_Q \subset Q$ for each $ Q\in \mathcal S$ so that (a)  $ \lvert  E_Q\rvert >c \lvert  Q\rvert $ for each $ Q\in \mathcal S$, and  (b) the collection of sets $\{ E_Q \;:\; Q\in \mathcal S\}$ are pairwise disjoint.  Here $ 0< c< 1$ will be a dimensional constant that we do not track.   
Define a \emph{sparse bilinear form} to be 
\begin{equation}\label{e:sparseOp}
\Lambda _{r,s} (f,g) = \sum_{Q\in \mathcal S}  \langle f \rangle _{Q,r} \langle g \rangle _{Q,s} \lvert  Q\rvert,   \qquad 1 \leq  r, s < \infty . 
\end{equation}
Above,  $ \langle  f \rangle _{Q,r} ^{r}  :=   \lvert 3 Q\rvert ^{-1} \int _{3Q} \lvert  f\rvert ^{r}    \;dx $, 
and if $ r=s$, then $ \Lambda _{r} = \Lambda _{r,r}$.  
We frequently suppress the collection of sparse cubes $ \mathcal S$.

We consider  Calder\'on-Zygmund singular integral operators  $ T $, defined  to be an $ L ^2 (\mathbb R ^{n})$ bounded convolution operator given by 
\begin{equation*}
\langle T f, g \rangle = \int\!\int K (x-y) f (y) g (x) \; dx\,dy . 
\end{equation*}
for   compactly supported functions $ f, g$ with disjoint supports.  
Moreover, the kernel $ K (y)$ satisfies 
\begin{align}\label{e:size}
\lvert \nabla ^{t} K (x,y)\rvert \leq C_t \lvert  x-y\rvert ^{-n-t}  , \qquad x\neq y \in \mathbb R ^{n}, 
\end{align}
for   $  t \in \{ 0,1\} $.   Key examples are $ K (y) = 1/y$ in dimension one, and the Riesz transform kernels $ y/ \lvert  y\rvert ^{n+1} $, in dimension $ n$.  

Such operators are of course non-local, and involve subtle cancellative effects. It is thus something of a surprise that such operators are dominated by sparse operators, which have none of these features.  This is a special case of \cites{150105818,14094351,150805639}.

\begin{priorResults}\label{t:czo}    For each Calder\'on-Zygmund singular integral  operator $ T$ and  bounded compactly supported function $ f$, there is a sparse operator $ \Lambda  = \Lambda _{T, f}$  so that 
$
\lvert  T f \rvert \lesssim \Lambda _{1} f 
$. 
\end{priorResults}

An immediate corollary are weighted inequalities that are sharp in the $ A_p$ characteristic. See \cites{MR3085756,150105818,14094351}.

We consider polynomials of a fixed degree $ d$, given by 
$  P  (x,y) = \sum_{\alpha , \beta  \;:\; \lvert  \alpha \rvert + \lvert  \beta \rvert \leq d } \lambda _{\alpha, \beta  } x ^{\alpha } y ^{\beta }$, 
 where we use the usual multi-index notation. 
 The polynomial modulated Calder\'on-Zygmund operators are 
\begin{equation} \label{e:TK}
T_P f (x) = \int e ^{i P  (x,y)}K (y) f (x-y) \; dy . 
\end{equation}
The $ L ^{p}$ result below  is a special case of the results of Ricci and Stein \cites{MR822187,MR890662},  and the weak-type result is due to Chanillo and Christ \cite{MR883667}.

\begin{priorResults}\label{t:fixedPoly}  For $ 1 < p < \infty $,  the operator $ T_P$ is bounded on $ L ^{p}$, that is 
\begin{equation*} 
\lVert T_P  \;:\;  L^{p} \mapsto L ^{p}\rVert  \lesssim  1, 
\end{equation*}
where the implied constant depends on the degree of $ P$, and in particular is independent of $ \lambda $.  
Moreover, $ T_P$  maps $ L ^{1}$ to weak $ L ^{1}$, with the same bound.  
\end{priorResults}

The dependence on the polynomial being felt only through the degree of $ P$ is important to the application of these bounds to the setting of nilpotent groups, like the Heisenberg group, see \cite{MR890662}.  This dependence continues to hold in the Theorems below.

\begin{theorem}\label{t:main1}  
For each $ 1< r < 2$  Calder\'on-Zygmund operator $ T$,   polynomial $ P = P(y)$ of degree $ d$ and bounded  supported functions $f, g $ there is a  bilinear form  $ \Lambda _{r} $ so that 
\begin{equation*}
\lvert  \langle T _{P} f ,g  \rangle\rvert  \lesssim \Lambda _{r} (f,g). 
\end{equation*}
The implied constant  depends only on $ T$, the degree $ d$, and dimension $ n$ and choice of $ r>1$. 
\end{theorem}

The bound above continues to hold for polynomials $ P $ of two variables, but we suppress the details, as the estimate above can 
most likely be improved. And, as written is quite easy to prove, yet yields a non-trivial corollary. 

\begin{corollary}\label{c:wtd}  For $ 1 < p < \infty $,  the operator $ T_P$, where $ P = P (y)$ is of degree $ d$, 
is bounded on $ L ^{p} (w)$,   
where $ w $ is a Muckenhoupt weight $ w\in A_p $.   
\end{corollary}

Weak-type and weighted estimates for oscillatory singular integrals have been studied in this and more general contexts by various authors, see for instance \cites{MR2900003, MR2910762, MR2949870, MR2115460,  MR1782909}.  
Y.~Ding and H.~Liu \cite{MR2900003} were interested in $ L ^{p} (w)$ inequalities for more general operators $ T$. 
The approach of these authors entails many complications.

The method of proof of Theorem~\ref{t:main1} is  very simple. And, so we suspect that stronger results are possible. 
For instance, this Conjecture would imply nearly sharp $ A_p$ bounds, for all $ 1< p < 2$.  
\begin{conjecture}\label{j:oscillatory} 
 For $ 1 < r < \infty $,  the operator $ T_P$, where $ P = P (y)$ is of degree $ d$, for each bounded compactly supported function $ f$, there is a sparse operator $ \Lambda _{1,r}$ so that 
 \begin{equation*}
\lvert  \langle T _{P} f ,g  \rangle\rvert  \lesssim \Lambda _{1,r} (f,g). 
\end{equation*}
\end{conjecture}

It seems likely that the weak type argument of Chanillo and Christ \cite{MR883667} would establish the Conjecture for $ r=2$. Also see \cite{160901564}.  

\bigskip 

We turn to weighted inequalities for 
\emph{discrete random Hilbert transforms} acting on functions on $ \ell ^2 (\mathbb Z )$. 
 Define a sequence of Bernoulli rvs 
$ \{ X _{n} \;:\; n\neq 0\}$ with $ \mathbb P (X_n =1) = \lvert  n\rvert ^{- \alpha } $, where $ 0 \leq \alpha < 1$. 
Then, the set $ \{n \;:\; X_n=1\}$ is a.s.~infinite, by the Borel-Cantelli Lemma. Then, we consider the random Hilbert transform, and  maximal function below.   
\begin{align}  \label{e:Ha}
H _{\alpha } f (x)& = \sum_{n\neq 0}  \frac {X_n} {n ^{1- \alpha }} f (x-n). 
\\ \label{e:Ma}
M _{\alpha }  f (x)& = \sup_{n > 0}   \Bigl\lvert \frac 1 { S_N} \sum_{n=1} ^{N} X_n f (x-n) \Bigr\rvert, \qquad S_N = \sum_{n=1} ^{N} X_n . 
\end{align}

Our sparse bound here is more restrictive, with the value of  the sparse index $ r$ depending upon random parameter $ \alpha $. 
\begin{theorem}\label{t:R} 
For any  $ 0< \alpha < 1$, $ 1+ \alpha < r < 2$, 
almost surely, the following holds: 
For all   functions  $ f, g$ finitely supported  on $ \mathbb Z $, there is 
a bilinear  sparse operator $ \Lambda _{r} $ so that 
\begin{equation*}
  \lvert  \langle H _{\alpha } f, g \rangle\rvert  \lesssim \Lambda _{r}  (f,g).  
\end{equation*}
The same inequality holds for $ M _{\alpha }$.  (The sparse operator can be taken non-random, but the implied constant is random.)  
\end{theorem}

Weighted inequalities are a corollary.  They are \emph{the  first we know of} holding for operators defined on sets of the integers with zero asymptotic density.

\begin{corollary}\label{c:R} For any  $ 0< \alpha < 1$, 
almost surely, the following holds:  For all $ 1 + \alpha < p <  \frac {1+ \alpha } \alpha $, and  weights $ w$ so that 
\begin{equation} \label{e:WW}
w ^{1+ \alpha } \in A _{ (1+ \alpha ) (p-1)+1 }, \qquad w   \in A _{1 + \frac 1 {(1+\alpha) (p'-1)}} , 
\end{equation}
we have 
$
\lVert  H _{\alpha } \;:\; \ell ^{p} (w) \mapsto \ell ^{p} (w)\rVert  < \infty  
$. 
The implied constant   only depends upon $ [ w ^{1+ \alpha }] _{A _{ (1+ \alpha ) (p-1)+1 }}$, 
and $ [w] _{A_{1 + \frac 1 {\alpha (p'-1)}} }$.  
The same inequality holds for $ M _{\alpha }$.  
\end{corollary}

The study of these questions was initiated by Bourgain \cite{MR937581}, as an elementary example of a sequence of integers for which one could derive $ \ell ^{p}$ inequalities, with the sequence of integers also having asymptotic density zero.  Various aspects of these questions have been studied, both in $ \ell ^{p}$, in the weak $ (1,1)$ endpoints 
\cites {MR1325697,MR2680392,MR3421994,MR2318621,MR2576702}.  We are not aware of any result in the literature that proves a weighted estimate in this sort of discrete setting.  (If the set of integers has full density, it is easy to transfer weighted estimates.) 
  
There is a subtle difference between the  Hilbert transform and the maximal function in this random setting. In particular, more should be true for the maximal function. Prompted by the work of LaVictoire \cite{MR2576702}, we pose 

\begin{conjecture}  For $ 0< \alpha < 1/2$, almost surely,  for all $ 1< r < 2$, and finitely supported functions $ f, g$, 
there is a sparse operator $ \Lambda _{1,r}$ so that 
\begin{equation*}
\langle M _{\alpha } f , g \rangle \lesssim \Lambda _{1,r} (f,g).  
\end{equation*}  
\end{conjecture}

\bigskip 
We turn to the   context for our paper.  The concept of sparse operators arose from Lerner's remarkable median inequality \cite{MR2721744}. It's application to weighted inequalities was advanced by several authors, with a  high point of this development being Lerner's argument \cite {MR3085756} showing that the weighted norm of Calder\'on-Zygmund operators is comparable to that of the norms of sparse operators.  This lead to the question of pointwise control, namely Theorem~\ref{t:czo}.  
First established by Conde-Alonso and Rey \cite{14094351}, also see Lerner and Nazarov \cite{150805639}, 
the author \cite{150105818}  established  Theorem~\ref{t:czo} with a stopping time argument. 
The latter argument  was extended by Bernicot, Frey and Petermichl  \cite{MR3531367} to a setting where the operators are generated by semigroups, including examples outside the scope of classical Calder\'on-Zygmund theory.   For closely related developments see \cites{MR3484688,151005789}.  The sparse bounds for commutators \cites{160401334,160405506} are remarkably powerful.  Edging beyond the Calder\'on-Zygmund context, Benau, Bernicot and Frey \cite{160506401} have supplied sparse bounds for certain Bochner-Riesz multipliers.  

Very recently, Culiuc, di Plinio and Ou \cite{160305317}  have established a sparse domination result in a setting far removed from the extensions above: The trilinear form associated to the  bilinear Hilbert transform is dominated by a sparse form.
This is a surprising result, as the bilinear Hilbert transform has all the difficult features of the Hilbert transform, with  additional oscillatory  and arithmetic-like aspects.  
This paper is  an initial effort on our part to understand how general a technique `domination by sparse' could be. 
There are plenty of additional directions that one could think about.  

For instance, the interest in the oscillatory singular integrals is driven in part by their application to singular integrals defined on nilpotent groups.   Implications of the  sparse bound  in this setting are unexplored.  

%

\bigskip 
There are two approaches to sparse bounds, the bilinear form method \cites{160305317}, and the use of the maximal truncation inequality \cite{150105818}.   We use neither approach. After applying the known sparse bounds for singular integrals,   for the remaining parts of the operator, there is a very simple interpolation argument which you can use in the bilinear setting. The notable point about the proofs are that they are quite easy, and yet deliver striking applications.

\section{Proof of Theorem~\ref{t:main1}} 

Our conclusion is invariant under dilations of the operator. Hence, we can proceed under the assumption that 
$ \lVert P\rVert = \sum_{\alpha } \lvert  \lambda _{\alpha }\rvert  =1$.   We can also assume that the polynomial $ P$ has no linear term, as it can be absorbed into the function $ f$.   Under these assumptions we prove 

\begin{theorem}\label{t:s} Let $ P$ be a polynomial  without linear terms, and $ \lVert P\rVert=1$. 
Then, for bounded compactly supported functions $ f, g$ and $ 1< r < \infty $, 
there is a sparse form $ \Lambda_1 $ and a $ \eta >0$ so that 
\begin{equation}\label{e:s}
\lvert  \langle T_P f, g \rangle\rvert \lesssim \Lambda _1 (f,g) 
+ \sum_{ Q  \in \mathcal D\;:\; \lvert  Q\rvert \geq 1 }  
\langle  f  \rangle _{Q,r} \langle g \rangle _{Q,r} \lvert  Q\rvert  ^{1- \eta }
\end{equation}
\end{theorem}

It is easy to see that this implies Theorem~\ref{t:main1}, since the second term on the right is restricted to dyadic cubes of volume at least one, and there is a gain of $ \lvert  Q\rvert ^{- \eta } $.  Moreover, we will see that this Theorem implies the  weighted result.

Let $ e (\lambda ) = e ^{i \lambda } $ for $ \lambda \in \mathbb R $.  
If the kernel  $ K$ of $ T$ is supported on $ 2B= \{y \;:\; \lvert  y\rvert \leq 2\}$, then we have 
\begin{equation*}
\lvert  e (P (y)) K (y) - K (y)\rvert \lesssim  \mathbf 1_{2B} (y) \lvert  y \rvert ^{-n+1},  
\end{equation*}
so that $ \lvert  T _{P} f - T f \rvert \lesssim M f $.  Both $ T$ and $ M$ admit pointwise domination by sparse forms, hence also by bilinear forms.  (This is the main result of  \cite{150105818}.)

Thus, we can proceed under the assumption that the kernel $ K $ is not supported on $ B$. We can then write 
\begin{equation*}
K  =  \sum_{j =1} ^{\infty  } \varphi _j  
\end{equation*}
where   $ \varphi _j$ is supported on $ 2 ^{j-1} B \setminus 2 ^{j-2}B$, with 
$ \lVert  \nabla ^{s} \varphi _ j\rVert _{\infty } \lesssim 2 ^{-nj -sj}$, for $ s=0,1$.

We use shifted dyadic grids, $ \mathcal D _{t}$, for $ 1\leq t \leq 3 ^{n}$. These grids have the property that 
\begin{equation*}
\{ \tfrac 13  Q  \;:\; Q\in \mathcal D_t,\  \ell Q= 2 ^{k}, 1\leq t \leq 3 ^{n}\}
\end{equation*}
form a partition of $ \mathbb R ^{n}$.   Throughout, $ \ell Q = \lvert  Q\rvert ^{1/n} $ is the side length of the cube $ Q$. 
We fix a dyadic grid $ \mathcal D_t$ throughout the remainder of the argument, 
and set $ \mathcal D_+ = \{Q \;:\; \ell Q > 2 ^{10}\}$.  Define  
\begin{equation*}
I _{Q} f = \int e (P  (y)) \varphi _{ k} (y)  (\mathbf 1_{\tfrac 13Q} f ) (x-y)\; dy, \qquad \ell Q = 2 ^{k+2}.  
\end{equation*}
Note  that $ I _Q f $ is supported on $ Q$, and that we have suppressed the dependence on $ P $, which we will continue below.

The basic estimate is then this Lemma. 

\begin{lemma}\label{l:fixed}  For each  cube $ Q$ with $ \lvert  Q\rvert \geq 1 $ and  $ 1< r < 2$, there holds 
\begin{equation}\label{e:fixed}
 \lvert  \langle I _{Q} f, g \rangle\rvert 
\lesssim 2 ^{- \eta k}  \langle f \rangle _{Q, r} \langle g \rangle _{Q,r} \lvert  Q\rvert , 
\end{equation}
where $ \eta = \eta (d,n,r) >0 $.  
\end{lemma}

Theorem~\ref{t:s} follows immediately from this Lemma.  
The oscillatory nature of the problem exhibits itself in the next Lemma. Write 
\begin{equation}  \label{e:K}
I ^{\ast} _Q I_Q  \phi (x) = \mathbf 1_{\tfrac 13Q}  (x)  \cdot \int _{\tfrac 13 Q}  K _Q (x,y) \phi  (y) \; dy . 
\end{equation}

\begin{lemma}\label{l:K} For each cube $ Q \in \mathcal D_+$, and $ x\in \tfrac 13 Q$, we have 
\begin{equation}\label{e:K<}
\lvert  K_Q (x,y)\rvert 
\lesssim    \lvert  Q\rvert ^{-1}  \mathbf 1_{Z_ {Q}} (x - y) +  \lvert  Q\rvert ^{-1 - \epsilon } \mathbf 1_{ Q} (x) \mathbf 1_{Q} (y),  
\end{equation}
where $ Z_ {Q} \subset   Q   $ has measure at most $  (\ell Q) ^{- \epsilon } \lvert Q\rvert   $, where  $ \epsilon = \epsilon (n,d)>0$. 
\end{lemma}
 
This Lemma is well known, see for instance \cite{MR1879821}*{Lemma 4.1}.   Here is how we use the Lemma.  
Using   Cauchy-Schwartz, we have 
\begin{align}\label{e:2}
\lVert I_Q f \rVert_2 ^2 &\lesssim    \lvert  Q\rvert ^{-1} \int _{Q} \int _{Z_Q} \lvert  f (x)\rvert  \lvert  f (x-y)\rvert  \; dy dx 
+  \lvert  Q\rvert ^{ - \epsilon } \langle f \rangle _{Q,1} ^{2} \lvert  Q\rvert 
\\&
\lesssim 
\lvert  Q\rvert ^{- \epsilon /n}    \lVert f \mathbf 1_{Q}\rVert   _{2 } ^2 . 
\end{align}
We also  have the trivial but rarely used   $ \lVert I _{Q} f \rVert_ \infty  \lesssim   \lvert  Q\rvert ^{-1}  \lVert f \mathbf 1_{Q}\rVert _1     $.   
By Riesz Thorin interpolation, there holds with $ \ell Q = 2 ^{k}$, 
\begin{equation*}
\lVert I _{Q} f \rVert _{r'} \lesssim 2 ^{- \eta  k}  \lvert  Q\rvert ^{-1+ 2/ r'}    \lVert f \mathbf 1_{Q}\rVert _r, \qquad 1 < r  \leq 2,\  r'= \tfrac r {r-1}.   
\end{equation*}
Above, $ \eta = \eta (\epsilon , r)$
But, this immediately implies  \eqref{e:fixed}.  Namely, 
\begin{align*}
\lvert  \langle I_Q f, g \rangle \rvert & \lesssim  
\lVert I _{Q} f \rVert _{r'}  \lVert g \mathbf 1_{Q}\rVert _{r} 
 \\&\lesssim 2 ^{- \eta  k}  \lvert  Q\rvert ^{-1+ 2/ r'}    \lVert f \mathbf 1_{Q}\rVert _r  \lVert g \mathbf 1_{Q}\rVert _{r}  
\\&=  2 ^{- \eta  k} \langle f \rangle _{Q,r} \langle g \rangle _{Q,r} \lvert  Q\rvert.  
\end{align*}
(Alternatively, one can just use bilinear interpolation.) 

We now give the weighted result. 

\begin{proof}[Proof of Corollary~\ref{c:wtd}] 
The qualitative result that $ T_P $ is bounded on $ L ^{p} (w)$ for $ w \in A_p$, $ 1< p < \infty $ is as follows. 
Given $ w \in A_p$, recall that the dual weight is $ \sigma = w ^{1- p'}$.  Then, it is equivalent to show that 
\begin{equation*}
\lvert  \langle T _{P}  (f \sigma ),   g w    \rangle\rvert \lesssim 
C _{[w] _{A_p}}\lVert f\rVert _{L ^{p} (\sigma) } \lVert g\rVert _{L ^{p'} (w)}. 
\end{equation*}
Using the sparse domination from \eqref{e:s}, we see that we need to prove the corresponding bound for the terms on the right in  \eqref{e:s}.  Now, it is well known \cite{MR3085756} that 
\begin{equation*}
\Lambda _{1} (   f, g ) \lesssim 
[w] ^{\max \{ 1, \frac {1} {p-1} \}} _{A_p} \lVert f\rVert _{L ^{p} (w ) } \lVert g\rVert _{L ^{p'} (w)}. 
\end{equation*}
Indeed, this is a key part of the proof of the $ A_2$ Theorem by sparse operators.  

So, it remains to consider the second term on the right in \eqref{e:s}. For each $ k \in \mathbb N $, we have by 
Proposition~\ref{p:scale}, $ k\in \mathbb Z $, 
\begin{equation}
\sum_{ Q  \in \mathcal D\;:\; \lvert  Q\rvert  =  2 ^{nk} }  
\langle  f     \rangle _{Q,r} \langle g \rangle _{Q,r} \lvert  Q\rvert 
\lesssim 
 [w] _{A_p} ^{1/p} [w] _{RH_r} [ \sigma ] _{RH_r}
 \lVert f\rVert _{L ^{p} (w ) } \lVert g\rVert _{L ^{p'} (w)} . 
\end{equation}
As we recall in \S~\ref{s:wtd}, there is a $ r = r ([w] _{A_p}) >1$ so that 
$  [w] _{RH_r}  [ \sigma ] _{RH_r} < 4 $.  And so the proof of the Corollary is complete.

Indeed, it is easy enough to make this step quantitative. For $ 2 < p < \infty $, the choice of $ r$ can be taken to satisfy $ r-1 > c [w] _{A_p} ^{-1} $, which then means that the choice of $ \eta = \eta (r)$ in \eqref{e:s} is at least as big is 
$ c [w] _{A_p} ^{-1} $.  Then, our bound is 
\begin{equation}  \label{e:sharp?}
\langle T_P ( \sigma  f), g w \rangle _{} \lesssim [w] ^{1+\frac 1p }  _{A_p}\lVert f\rVert _{L ^{p} (\sigma )} \lVert g\rVert _{L ^{p'} (w)}, \qquad 2 < p < \infty . 
\end{equation}
We have no reason to believe that this estimate is sharp.  
\end{proof}

\section{Random Hilbert Transforms} 

The discrete Hilbert transform 
\begin{equation*}
H f (x) = \sum_{n\neq 0} \frac {f (x-n)} {n} 
\end{equation*}
satisfies a sparse bound: For all finitely supported functions $ f$ and $ g$, there is a sparse operator $ \Lambda $ so that 
\begin{equation} \label{e:dHS}
\lvert \langle H f, g \rangle\rvert \lesssim \Lambda _{1,1} (f,g).  
\end{equation}
This is a consequence of the main results of  Theorem~\ref{t:czo}.  
Recalling the definition of $ H _{\alpha } $ in \eqref{e:Ha}, we see that $ \mathbb E H_ \alpha  f  = Hf$, so it remains to consider the difference 
\begin{align*}
 H_ \alpha  f (x) - H f (x) &:= \sum_{k=1} ^{\infty } 
 \sum_{ n \;:\; 2 ^{k-1} \leq \lvert  n\rvert < 2 ^{k} }  \frac { X_n - n ^{- \alpha }} {n ^{1- \alpha } } f (x-n) 
\\&
:= \sum_{k=1} ^{\infty }  T_k f (x). 
\end{align*}
Above, we have passed directly to the distinct scales of the operator.  
We will subsequently write $ Y_n =  X_n - n ^{- \alpha }$, which are independent mean zero random variables.  

The crux of the matter are these two estimates: 

\begin{lemma}\label{l:almost} Almost surely, for all $ 0< \epsilon < 1$, and 
for all integers $ k$, and  $ f, g$ supported on an interval $ I$ of length $ 2 ^k$,  we have 
\begin{equation}\label{e:almost}
\lvert  \langle   T_k f, g \rangle\rvert 
\lesssim 
\begin{cases}
2 ^{- k \frac {1- \alpha  } 2 + \epsilon  }  \langle f  \rangle _{I,2} \langle g \rangle _{I,2} \lvert  I\rvert 
\\
2 ^{k \alpha } \langle f  \rangle _{I,1} \langle g \rangle _{I,1} \lvert  I\rvert 
\end{cases}. 
\end{equation}
The implied constant is random, but independent of $ k \in \mathbb N $ and the choice of functions $ f,g $.   
\end{lemma}

\begin{proof}
The second bound follows trivially from      
$   \lvert  Y_n\rvert /n ^{1- \alpha }  \mathbf 1_{   2 ^{k-1} \leq \lvert  n\rvert < 2 ^{k} }  \lesssim 2 ^{k (\alpha-1)} $.   
For the first bound, we clearly have 
\begin{equation*}
\lvert  \langle   T_k f, g \rangle\rvert 
\leq \lVert T_k \;:\; \ell ^2 \to \ell ^2 \rVert \cdot  \langle f  \rangle _{I,2} \langle g \rangle _{I,2} \lvert  I\rvert , 
\end{equation*}
so it suffices to estimate the operator norm above.  The assertion is that with high probability, the operator norm is small: 
\begin{equation}\label{e:small}
\mathbb P  \bigl( \lVert T_k \;:\; \ell ^2 \to \ell ^2 \rVert  > C \sqrt k 2 ^{- k \frac {1- \alpha  } 2  }  \bigr) \lesssim  2 ^{-k}, 
\end{equation}
provided $ C$ is sufficiently large.   Combine this with the Borel Cantelli Lemma to prove the Lemma as stated.  

\smallskip 
By Plancherel's Theorem, the operator norm is equal to  $ \lVert Z (\theta )\rVert _{L ^{\infty } ( d\theta )}$, where 
\begin{equation*}
Z (\theta ) :=  
\sum_{ n \;:\; 2 ^{k} \leq \lvert  n\rvert < 2 ^{k+1} }  Y_n \frac {   e ^{ 2 \pi i \theta }} {n ^{1- \alpha } }. 
\end{equation*}
The expression above is a random Fourier series, with frequencies at most $ 2 ^{k+2}$.  By Bernstein's Theorem for trigonometric polynomials, the $ L ^{\infty } (d \theta )$ norm can be estimated by testing the norm on at most $ 2 ^{k+3}$ equally spaced points in $ \mathbb T $, that is, we have 
\begin{align*}
\mathbb P  \bigl( \lVert Z (\theta )\rVert_ \infty >  C \sqrt k 2 ^{- k \frac {1- \alpha  } 2  } \bigr)
\lesssim 2 ^{k} 
\sup _{\theta } 
\mathbb P  \bigl( \lvert Z (\theta )\rvert  >  C \sqrt k 2 ^{- k \frac {1- \alpha  } 2  } \bigr), 
\end{align*}
 where we have simply used the union bound.  
 
Now, $ Z (\theta )$ is the sum of independent, mean zero  random variables, 
which are bounded by one, and have standard deviation bounded by  $ c 2 ^{ - k  \frac {1- \alpha  } 2}$. 
So by, for instance, the Bernstein inequality, it follows that 
 \begin{equation*}
\mathbb P  ( \lvert Z (\theta )\rvert  >  C \sqrt k 2 ^{- k \frac {1- \alpha  } 2  } ) \lesssim 2 ^{-2k}, 
\end{equation*}
for appropriate $ C$. This completes the proof. 
\end{proof}

From the previous Lemma, we have the Corollary below. It with the sparse bound for the Hilbert transform 
\eqref{e:dHS} completes the proof of Theorem~\ref{t:R}, for the random Hilbert transform. The case for maximal averages is entirely similar.  

\begin{corollary}\label{c:almost} Almost surely, for $  {1+ \alpha } < r < 2$, there is a $ \eta >0$ so that  
for all integers $ k$, and all   functions $ f, g$ supported on an interval $ I$ of length $ 2 ^{k}$, we have 
\begin{equation}\label{e:calmost}
\lvert  \langle   T_k f, g \rangle\rvert \lesssim 
2 ^{-  \eta k   }  \langle f  \rangle _{I,r} \langle g \rangle _{I,r} \lvert  I\rvert . 
\end{equation}

\end{corollary}

\begin{proof}
This follows from Lemma~\ref{l:almost} by interpolation. 
The relevant interpolation parameter $ \theta _0$ at which we have  only an epsilon loss  in the interpolated estimate is given by 
\begin{align*}
(1-\theta_0)\alpha & = \theta_0 \frac {1- \alpha } 2 , 
\\
\textup{and then} \quad 
\frac 1 {r_0} &= \frac {1-\theta _0} 1 + \frac {\theta _0} 2 . 
\end{align*}
We see that $ r_0 = 1+ \alpha $. And so we conclude that for $ r_0 = 1+ \alpha < r < 2$, we have the required 
gain in the interpolated bound, which proves the Corollary. 
\end{proof}

We now turn to the weighted inequalities of Corollary~\ref{c:R}.  

\begin{proof}[Proof of Corollary~\ref{c:R}] 
For the deterministic Hilbert transform, we have the sharp bound of Petermichl \cite{MR2354322}, namely 
\begin{equation*}
\lVert H \;:\; \ell ^{p} (w) \mapsto \ell ^{p} (w)\rVert \lesssim [w] _{A_p} ^{ \max \{1, \frac 1 {p-1}\}}. 
\end{equation*}
So, it remains to bound the terms in \eqref{e:calmost}. By Proposition~\ref{p:scale}, we then need to see that 
the  hypotheses on $ w$, namely  \eqref{e:WW},  imply that for some choice of $ r > 1+ \alpha $, we have 
\begin{equation}\label{e:ww}
w \in A_p, \quad w \in RH _{r}, \quad  \sigma = w ^{1- p'} \in RH _{r}.  
\end{equation}

Recall that $ v \in A_q \cap RH _{s} $ if and only if $ v ^{s} \in A _{s(q-1)+1}$.  
Now, by assumption, $ w  ^{1+ \alpha }\in A _{ (1+ \alpha ) (p-1)+1}$.  So, there is a $ t>1$ so that 
$  w  ^{t(1+ \alpha )}\in A _{ (1+ \alpha ) (p-1)+1}$, and the $ A_q$ classes increase in $ q$, so we conclude that 
$ w \in A_p \cap  RH _{r} $, for a $ r > 1+ \alpha $.  

The second hypothesis is $ w  \in A _{1 + \frac 1 {(1+\alpha) (p'-1)}} $. This is equivalent to 
\begin{equation*}
(w ^{(1-p')}) ^{ 1+\alpha } \in A _{ (1+\alpha ) (p'-1) +1 }. 
\end{equation*}
Now, $ w ^{1- p'} = \sigma $ is the dual weight. So by the argument in the previous paragraph, $ \sigma \in RH _{r}$, for some $ r > 1+ \alpha $.  So the proof is complete. 
\end{proof}

\section{Sparse Bounds and Weighted Inequalities} \label{s:wtd}

Let us recall the weighted estimates that we need for our corollaries.  A  function $ w >0$ is a \emph{Muckenhoupt $ A_p$ weight}  if 
\begin{equation*}
[w] _{A_p} = \sup _{Q} \Bigl[\frac { w ^{\frac 1 {1-p}} (Q) } {\lvert  Q\rvert }\Bigr] ^{p-1} \frac {w (Q)} {\lvert  Q\rvert  } < \infty . 
\end{equation*}
Above, we are conflating $ w$ as a measure and a density, thus $w ^{\frac 1 {1-p}} (Q)  =\int _{Q} w (x) ^{\frac 1 {1-p}} \;dx $.  
 We have these estimates, which are sharp in the $ A_p$ characteristic.  They are an element  of the sparse proof of the $ A_2$ conjecture.  (See \cite{MR3085756} for a proof.)

\begin{priorResults}\label{t:wts}  These estimates hold for all $ 1  < p < \infty $. 
\begin{equation*}
 \lVert \Lambda _{1,1} \;:\; L ^{p} (w) \mapsto L ^{p} (w) \rVert  \lesssim [w] _{A _{p}} ^{\max \{1, \frac 1 {p-1}\}}  .
\end{equation*}

\end{priorResults}

For our applications, we have a second class of operators, a simplified form of those introduced by Benau-Bernicot-Petermichl 
\cite{160506401}.  For our purposes, we need a much simplified version of their result.  Define an additional characteristic of a weight, namely the \emph{reverse H\"older} property.  
\begin{equation}\label{e:RH}
[w] _{RH_r} = \sup _{ Q}   \frac {\langle w  \rangle _{Q, r}} {\langle w \rangle _{Q}}.  
\end{equation}

\begin{proposition}\label{p:scale} Fix an integer $ k$, and $ 1< r < 2$.  We have the bound below for all 
$ w\in A_p$, where $ r \leq p \leq r' = \frac r {r-1}$.  
\begin{equation}\label{e:scale}
\sum_{ Q \in \mathcal D \;:\; \lvert  Q\rvert ={2 ^{nk}} } 
\langle f \rangle _{Q,r} \langle g \rangle _{Q,r} \lvert  Q\rvert 
\lesssim [w] _{A_p} ^{1/p} [w] _{RH_r} [ \sigma ] _{RH_r} \lVert f\rVert _{L ^{p} (w)} \lVert g\rVert _{L ^{p'} (w)} 
\end{equation}
where $ \sigma = w ^{1 - p'}$ is the `dual' weight to $ w$.  
\end{proposition}

Let us recall these well known facts.  
\begin{enumerate}
\item  We always have $ [w] _{A_p}, [w] _{RH_r} \geq 1$. 

\item   For $ w\in A_p$ and $ \sigma = w ^{1-p'}$, the weight $ \sigma $ is locally finite, its `dual' weight is $ w$,  and $ [\sigma ] _{A _{p'} } = [w ] _{A_p} ^{p'-1}$. 

\item For every $ w \in A _{p}$ there is a 
$ r = r ([w] _{A_p}) >1$ so that $ w \in RH _{r}$. (In particular, we can take $ r$ so that $ r -1 \simeq [w] _{A_p} ^{-1} $, 
by \cite{MR3092729}*{Thm 2.3}. ) 

\item For every $ w \in A_p$, there is a $ r = r ([w] _{A_p}) >1$ so that $ w ^{r} \in A_p$.  

\item  We have $ w \in A_p \cap RH_r $ if and only if $ w ^{r} \in A _{r (p-1)+1}$, 
by \cite{MR1018575}.  
\end{enumerate}

\begin{proof}[Proof of Proposition~\ref{p:scale}] 
This inequality is rephrased in the self-dual way, namely setting $ \sigma = w ^{1-p'}$, it is equivalent to show that  for $ k\in \mathbb Z $, 
\begin{equation}\label{e:ss} 
\sum_{ \substack{Q  \in \mathcal D \\ \lvert  Q\rvert  =  2 ^{nk} }  }
\langle  f    \sigma  \rangle _{Q,r} \langle g w  \rangle _{Q,r} \lvert  Q\rvert 
\lesssim 
[w] ^{ \frac1p } _{A_p}  [ \sigma ] _{RH_r} [ w ] _{RH_r} \lVert f\rVert _{L ^{p} (\sigma ) } \lVert g\rVert _{L ^{p'} (w)}.
\end{equation}

Fix the integer $ k$.  We can assume that for $ \lvert  Q\rvert= 2 ^{nk} $, if $ f$ is not zero on $ Q$, then $ f \mathbf 1_{3Q \setminus Q} \equiv 0$, and we assume the same for $ g$. Then, set 
\begin{equation*}
f' =  \sum_{ Q  \in \mathcal D\;:\; \lvert  Q\rvert  =  2 ^{nk} }    \mathbf 1_{Q}  
\Bigl[
\frac 1 { \sigma (Q) } \int _{Q}  \lvert  f\rvert ^{r}  \; d \sigma 
\Bigr] ^{1/r}
\end{equation*}
and likewise for $ g'$.  It is  immediate that  $ \lVert f'\rVert _{L ^{p} ( \sigma )} \lesssim  \lVert f\rVert _{L ^{p} (\sigma )}$, 
thus in \eqref{e:ss}, it suffices to assume that $ f = f'$. Then, we can even assume that $ f $ and $ g$ are supported on a single cube $ Q$, and take the value 1 on that cube. 

Then, write 
\begin{align*}
\langle    \sigma \mathbf 1_{Q} \rangle _{Q,r} &\langle  w \mathbf 1_{Q}\rangle _{Q,r} \lvert  Q\rvert 
\leq 
[ \sigma ] _{RH_r} [ w ] _{RH_r} 
\langle    \sigma  \mathbf 1_{Q} \rangle _{Q,1} \langle  w \mathbf 1_{Q}\rangle _{Q,1} \lvert  Q\rvert 
\\
& \leq 
[ \sigma ] _{RH_r} [ w ] _{RH_r} 
\langle    \sigma \mathbf 1_{Q} \rangle _{Q,1}  ^{1/p'}\langle  w \mathbf 1_{Q}\rangle _{Q,1}  ^{1/p} 
 \cdot \sigma (Q) ^{1/p} w (Q) ^{1/p'} 
 \\
 & \leq  [ \sigma ] _{RH_r} [ w ] _{RH_r} [w] _{A_p} ^{1/p}\sigma (Q) ^{1/p} w (Q) ^{1/p'} .  
\end{align*}
This is the inequality claimed. 

\end{proof}

\end{document}